\documentclass[12pt,draftcls,onecolumn]{IEEEtran}
\usepackage{latexsym, epsfig, amsthm, amssymb, amsmath}
\usepackage{showidx}
\usepackage{tikz}

\title{ A Novel Description of Linear Time--Invariant  \\ Networks 
 via Structured Coprime Factorizations}

\author{ \c{S}erban Sab\u{a}u, \thanks{\c{S}erban Sab\u{a}u
is with the Electrical and Systems Engineering Dept.,
University of Pennsylvannia
email: serbans@seas.upenn.edu.}
Cristian Oar\u{a}, \thanks{Cristian Oar\u{a} is with the Faculty of the Automatic Control Dept., University ``Politehnica'' Bucharest. email: cristian.oara@acse.pub.ro}
Sean Warnick \thanks{ Sean Warnick is with the Information and Decision Algorithms Laboratories and the Faculty of the Computer Science Department at Brigham Young University.  email: sean.warnick@gmail.com.} 
and Ali Jadbabaie \thanks{Ali Jadbabaie is with the Faculty of the Electrical and Systems Engineering Dept.,  
University of Pennsylvannia
email: jadbabai@seas.upenn.edu.}
{\thanks{ C.~Oar\u{a} was supported by a grant of the Romanian National Authority for Scientific
Research, CNCS Ð UEFISCDI, project number PN-II-ID-PCE-2011-3-0235}}
{\thanks{A.~Jadbabaie was supported by AFOSR Complex Networks Program}}}

\renewcommand{\tilde}{\widetilde}

\newcommand{\FF}{{{\rm I \kern -0.2em R}}}
\newcommand{\RR}{{{\rm I \kern -0.2em R}}}

\newcommand{\CC}{{{\mbox{\rm \hspace*{0.05ex}
\rule[.18ex]{.18ex}{1.24ex} \kern -.65em C}}}} %\newcommand{\T}{{{\rm I}\hspace{-1mm}{\rm T}}}
\newcommand{\rank}{\operatorname{rank}}

\newcommand{\bea}{\begin{eqnarray}}
\newcommand{\eea}{\end{eqnarray}}

\newcommand{\wti}{\widetilde}

\newtheorem{theorem}{Theorem}[section]
\newtheorem{rem}[theorem]{Remark}
\newtheorem{prop}[theorem]{Proposition}
\newtheorem{lem}[theorem]{Lemma} 
\newtheorem{defn}[theorem]{Definition}
\newtheorem{assumption}[theorem]{Assumption}

\newcommand{\ba}{\left[ \begin{array}}
\newcommand{\baa}{\begin{array}}
\newcommand{\ea}{\end{array} \right]}
\newcommand{\eaa}{\end{array}}

\newcommand{\be}{\begin{equation}}
\newcommand{\ee}{\end{equation}}
\newcommand{\bb}{\begin{equation}\label}

\newcommand{\Z}{{\bf Z}}

\newcommand{\rf}[1]{(\ref{#1})}
\newcommand{\boC}{\CC}

\newcommand{\la}{\lambda}

\def\math#1{\ifmmode{#1} \else {$#1$}\fi}

\newcommand{\U}{{\cal U}}

\newcommand{\sg}{\ifmmode \Sigma \else $\Sigma$ \fi}

\date{}

\begin{document}
\maketitle

\begin{abstract} 
In this paper we study state--space realizations of Linear and Time--Invariant (LTI) systems.  Motivated by biochemical reaction networks,  Gon\c{c}alves and Warnick have  recently introduced the notion of a  {\em Dynamical Structure Functions} (DSF), a particular factorization of the system's transfer function matrix that elucidates the interconnection structure in dependencies between manifest variables. We build onto this work by showing an intrinsic connection between a DSF and certain sparse left coprime factorizations. By establishing this link, we provide an interesting systems theoretic interpretation of sparsity patterns of coprime factors. In particular we show how the sparsity of these coprime factors allows for a given LTI system to be implemented as a network of LTI sub--systems. We examine possible applications in distributed control such as  the design of a LTI controller that can be implemented over a network with a pre--specified topology.
\end{abstract}

 \section{Introduction}

Distributed and decentralized control of LTI systems has been a topic of intense research focus in control theory for more than 40 years. Pioneering work includes includes that of Radner \cite{Radner}, who revealed the sufficient conditions under which the minimal quadratic cost for a linear system can be achieved by a linear controller. Ho and Chu \cite{Ho}, laid the foundation of team theory by introducing a general class of distributed structures,
dubbed  {\em partially nested}, for which they showed the optimal LQG controller to be linear. More recently in \cite{Voul1, Voul2, Rotko, Shah} important advances were made for the case where the decentralized  nature of the problem is modeled as sparsity constraints on the input-output operator (the transfer function matrix) of the controller. These types of constraints are equivalent with computing the output feedback control law while having access to only partial measurements. Quite different from this scenario, in this work we are studying the meaning of sparsity constraints on the left coprime factors of the controller, which is not noticeable on its transfer function. In particular, we show how the sparsity of these coprime factors allows for the given LTI controller to be implemented over a LTI network with a pre--specified topology.

More recently, network reconstruction of biochemical reaction networks have motivated a careful investigation into the nature of systems and the many interpretations of {\em structure} or sparsity structure one may define \cite{Sean08}.  In this work, a novel partial structure representation for Linear Time--Invariant (LTI) systems, called the {\em Dynamical Structure Function} (DSF) was introduced.  The DSF was shown to be a factorization of a system's transfer function that represented the {\em open-loop causal dependencies among manifest variables}, an interpretation of system structure dubbed the {\em Signal Structure}. 

\subsection{An Introductory Example \cite{Sean08}}

%\subsection{The Dynamical Structure Function of a ``Ring'' or ``Delta'' Network}

%The DSFs main feature is in the ability to retain the interconnections structure of the component sub--systems  (such as cascade, series, feedback or star, delta and ring networks and combinations of such configurations).   

One important characteristic of the DSF is its ability to represent the impact that observed variables have on each other.  This can often effectively describe the interconnection structure between component subsystems within a given system. Consider for example the 3--hop ring (also called ``delta'') network in Figure~\ref{Ring}, where all the $Q(s)$  and $P(s)$ blocks represent  transfer functions of continuous--time LTI systems. We denote with $L(s)$ the transfer function from the input signals $U(s)$ to the outputs $Y(s)$.  By directly inspecting the signal flow graph in Figure~\ref{Ring} we can write the algebraic equations:

  \begin{small}
  \begin{equation} \label{RingRing}
\ba{c} Y_1(s) \\ Y_2(s) \\Y_3(s) \ea  = \ba{ccc} O & O & Q_{13}(s) \\ 
Q_{21}(s) & O  & O\\ O &  Q_{32}(s) &O\ea \ba{c} Y_1(s) \\ Y_2(s) \\ Y_3(s) \ea + \ba{ccc} I & O&O\\ 
O & P_{22}(s) &O\\ O&O&P_{33}(s) \ea  \ba{c} U_1(s) \\ U_2(s)\\U_3(s) \ea 
\end{equation}
\end{small}

We make the additional notation
\begin{small}
\begin{equation} \label{lokala1}
Q(s)\overset{def}{=}\ba{ccc} O & O & Q_{13}(s) \\ 
Q_{21}(s) & O  & O\\ O &  Q_{32}(s) &O\ea \quad \mathrm{and} \quad P(s)\overset{def}{=}  \ba{ccc} I & O&O\\ 
O & P_{22}(s) &O\\ O&O&P_{33}(s) \ea
\end{equation}
\end{small}

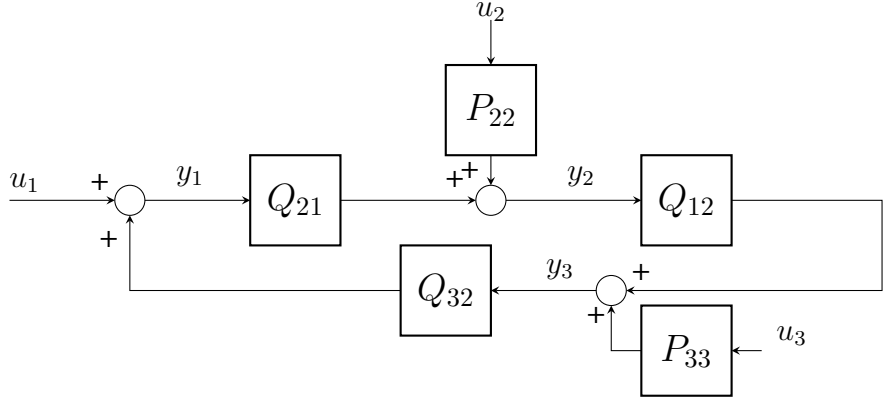
\begin{figure}
\begin{tikzpicture}[scale=0.4] 
\draw[xshift=5cm, >=stealth ] [->] (0,0) -- (3.5,0); 
\draw[ xshift=5cm ]  (4,0) circle(0.5); 
\draw[xshift=5cm] (3,0.6)   node {\bf{+}} (0.5,0.6) node {$u_1$};
\draw [xshift=5cm](6,0.8)   node {$y_1$} ;
\draw[ xshift=5cm,  >=stealth] [->] (4.5,0) -- (8,0); 
\draw[ thick, xshift=5cm]  (8,-1.5) rectangle +(3,3);
\draw [xshift=5cm](9.5,0)   node {\large{$Q_{21}$}} ;
\draw[ xshift=5cm,  >=stealth] [->] (11,0) -- (15.5,0); 
\draw[ xshift=5cm ]  (16,0) circle(0.5cm); 
\draw [xshift=5cm](19,0.8)   node {$y_2$} ;
\draw [xshift=5cm] (14.8,0.7)   node {\bf{+}};
\draw[  xshift=5cm,  >=stealth] [->] (16,6) -- (16,4.5); 
\draw[ thick, xshift=5cm ]  (14.5,1.5) rectangle +(3,3) ; 
\draw [xshift=5cm] (16,3)   node {\large{$P_{22}$}} ;
\draw[  xshift=5cm,  >=stealth] [->] (16,1.5) -- (16,0.5); 
\draw [xshift=5cm] (16,5.6)  node[anchor=south] {$u_2$}  (15.3,1)  node {\bf{+}};
\draw[  xshift=5cm,  >=stealth] [->] (16.5,0) -- (21,0); 
\draw[ thick, xshift=5cm ]  (21,-1.5) rectangle +(3,3) ; 
\draw [xshift=5cm] (22.5,0)   node {\large{$Q_{12}$}} ;
%\draw[  xshift=11cm,  >=stealth] [->] (19,0) -- (19,4) -- (21,4); 
%\draw[ thick,xshift=11cm ]  (21,2.5) rectangle +(3,3) ; 
%\draw [xshift=11cm] (22.5,4)   node {\large{$G_s$}} ;
%\draw[ xshift=5cm,  >=stealth] [->] (24,0) -- (31,0);
%\draw[ xshift=11cm ]  (27,0) circle(0.5cm); 
%\draw[xshift=11cm,  >=stealth] [->] (24,4) -- (27,4) -- (27,0.5); 
%\draw[xshift=11cm] (25.7,0.7)   node {\bf{+}}  (26.4,1.7)   node {\bf{+}};
%\draw[ xshift=11cm,  >=stealth] [->] (27.5,0) -- (31,0); 
\draw [xshift=5cm] (18.3,-2.3)   node {\bf{$y_3$}};
\draw[ xshift=5cm ]  (20,-3) circle(0.5cm); 
\draw [xshift=5cm] (19.5,-3.8)   node {\bf{+}};
\draw [xshift=5cm] (21,-2.4)   node {\bf{+}};
\draw [xshift=5cm] (26,-4.5)   node {\bf{$u_3$}};
\draw[ xshift=5cm,  >=stealth] [->]  (25,-5)--(24,-5);  %
\draw[ xshift=5cm,  >=stealth] [->]  (21,-5)--(20,-5)--(20,-3.5); 
\draw[ thick, xshift=5cm ]  (21,-6.5) rectangle +(3,3) ; 
\draw [xshift=5cm] (22.5,-5)   node {\large{$P_{33}$}} ;
\draw[ xshift=5cm,  >=stealth] [->] (24,0)--(29,0)--(29,-3)--(20.5,-3); 
\draw[ xshift=5cm,  >=stealth] [->] (19.5,-3)--(16,-3);
\draw[ xshift=5cm,  >=stealth] [->] (13,-3) --(4,-3) -- (4, -0.5);
\draw[ thick, xshift=5cm ]  (13,-4.5) rectangle +(3,3) ; 
\draw [xshift=5cm] (14.5,-3)   node {\large{$Q_{32}$}} ;
\draw [xshift=5cm] (3.3,-1.3)   node {\bf{+}};
%\draw [xshift=11cm] (30,1)   node {$y_G$};
\useasboundingbox (0,5);
\end{tikzpicture} 
\caption{A 3--Hop Ring Network}
\label{Ring} 
\end{figure}

 \noindent and we define ad-hoc the $\big( Q(s),P(s) \big)$ pair to be the Dynamical Structure Function associated with the $L(s)$ LTI system. (The rigorous definition of DFS will be introduced in Section~\ref{adoua} following the original mathematical derivation from \cite{Sean08}.) An interesting observation, which is the main thesis of this work, is that the structure of the subsystems interconnections  in Figure~\ref{Ring} is no longer recognizable from the input-output relation described by the transfer function of the aggregate system $\displaystyle  L(s)=\big( I_3 - Q(s) \big)^{-1}P(s)$  since  the transfer function $L(s)$ does not have any sparsity pattern and in general does not have any other particularities.  The structure however, remains visible and it is captured in the quite particular sparsity patterns of $Q(s)$ and $P(s)$, respectively. This key property makes the DSF susceptible of becoming a perfectly suited theoretical concept to model any LTI network.

We want to illustrate further  how  the DSF  determines via equation (\ref{DSF}) the topology of the LTI network that can describe the given LTI system $L(\la)$.  If we consider $Q_{13}(s)$ identically zero in (\ref{Ring}) which would mean``breaking'' the ring network from Figure~\ref{RingRing}  then it becomes a cascade connection and $L(\la)$ can be implemented as a ``line'' network. A ``line'' network controller could be interesting for example  motion control of vehicles moving in a platoon formation.

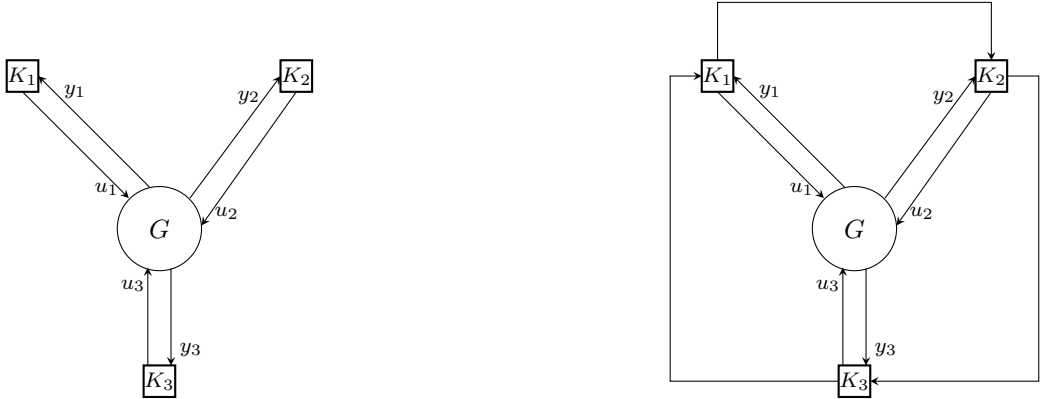
\begin{figure}
\begin{tikzpicture}[scale=0.14] 
\draw[ thick, xshift=44 cm]  (3,-1.5) rectangle +(3,3);
\draw [xshift=44 cm](4.5,0)   node {$\mbox{\fontsize{8}{10}\selectfont $K_1$}$} ;
\draw[xshift=44 cm, >=stealth ] [->] (4.5,-1.5) -- (14.6,-11.6); 
\draw [xshift=44 cm](12.5,-10.7)   node {$\mbox{\fontsize{8}{10}\selectfont $u_1$}$} ;
\draw[xshift=44 cm, >=stealth ] [->] (16.6,-10.6) -- (6,0); 
\draw [xshift=44 cm](9.5,-1.5)   node {$\mbox{\fontsize{8}{10}\selectfont $y_1$}$} ;
\draw[xshift=44 cm ]  (17.5,-14.5) circle(4); 
\draw [xshift=44 cm](17.5,-14.5)   node {$\mbox{\fontsize{10}{12}\selectfont $G$}$} ;
%%%%%%%%%%%%%%%%%%%%%%%%%%%%%%%%%%%%%%%%%%
\draw[ thick, xshift=44 cm]  (29,-1.5) rectangle +(3,3);
\draw [xshift=44 cm](30.5,0)   node {$\mbox{\fontsize{8}{10}\selectfont $K_2$}$} ;
\draw[xshift=44 cm, >=stealth ] [->] (20.4,-11.6)--(29,0); 
\draw [xshift=44 cm](26,-2)   node {$\mbox{\fontsize{8}{10}\selectfont $y_2$}$} ;
\draw[xshift=44 cm, >=stealth ] [->] (30.5,-1.5)--(21.5,-14.2); 
\draw [xshift=44 cm](23.9,-13)   node {$\mbox{\fontsize{8}{10}\selectfont $u_2$}$} ;
%%%%%%%%%%%%%%%%%%%%%%%%%%%%%%%%%%%%%%%%%%
\draw[ thick, xshift=44 cm] (16,-30.5) rectangle +(3,3);
\draw [xshift=44 cm](17.5,-29)   node {$\mbox{\fontsize{8}{10}\selectfont $K_3$}$} ;
\draw[xshift=44 cm, >=stealth ] [->] (16.4,-27.5)--(16.4,-18.2); 
\draw [xshift=44 cm](15,-19.9)   node {$\mbox{\fontsize{8}{10}\selectfont $u_3$}$} ;
\draw[xshift=44 cm, >=stealth ] [->] (18.6,-18.3)--(18.6,-27.5); 
\draw [xshift=44 cm](20.5,-26)   node {$\mbox{\fontsize{8}{10}\selectfont $y_3$}$} ;
%%%%%%%%%%%%%%%%%%%%%%%%%%%%%%%%%%%%%%%%%%%%
%%%%%%%%%%%%%%%%%%%%%%%%%%%%%%%%%%%%%%%%%%%%
%%%%%%%%%%%%%%%%%%%%%%%%%%%%%%%%%%%%%%%%%%%%
%%%%%%%%%%%%%%%%%%%%%%%%%%%%%%%%%%%%%%%%%%%%
%%%%%%%%%%%%%%%%%%%%%%%%%%%%%%%%%%%%%%%%%%%%
\draw[ thick, xshift=110 cm]  (3,-1.5) rectangle +(3,3);
\draw [xshift=110 cm](4.5,0)   node {$\mbox{\fontsize{8}{10}\selectfont $K_1$}$} ;
\draw[xshift=110 cm, >=stealth ] [->] (4.5,-1.5) -- (14.6,-11.6); 
\draw [xshift=110 cm](12.5,-10.7)   node {$\mbox{\fontsize{8}{10}\selectfont $u_1$}$} ;
\draw[xshift=110 cm, >=stealth ] [->] (16.6,-10.6) -- (6,0); 
\draw [xshift=110 cm](9.5,-1.5)   node {$\mbox{\fontsize{8}{10}\selectfont $y_1$}$} ;
\draw[xshift=110 cm ]  (17.5,-14.5) circle(4); 
\draw [xshift=110 cm](17.5,-14.5)   node {$\mbox{\fontsize{10}{12}\selectfont $G$}$} ;
%%%%%%%%%%%%%%%%%%%%%%%%%%%%%%%%%%%%%%%%%%
\draw[ thick, xshift=110 cm]  (29,-1.5) rectangle +(3,3);
\draw [xshift=110 cm](30.5,0)   node {$\mbox{\fontsize{8}{10}\selectfont $K_2$}$} ;
\draw[xshift=110 cm, >=stealth ] [->] (20.4,-11.6)--(29,0); 
\draw [xshift=110 cm](26,-2)   node {$\mbox{\fontsize{8}{10}\selectfont $y_2$}$} ;
\draw[xshift=110 cm, >=stealth ] [->] (30.5,-1.5)--(21.5,-14.2); 
\draw [xshift=110 cm](23.9,-13)   node {$\mbox{\fontsize{8}{10}\selectfont $u_2$}$} ;
%%%%%%%%%%%%%%%%%%%%%%%%%%%%%%%%%%%%%%%%%%
\draw[ thick, xshift=110 cm] (16,-30.5) rectangle +(3,3);
\draw [xshift=110 cm](17.5,-29)   node {$\mbox{\fontsize{8}{10}\selectfont $K_3$}$} ;
\draw[xshift=110 cm, >=stealth ] [->] (16.4,-27.5)--(16.4,-18.2); 
\draw [xshift=110 cm](15,-19.9)   node {$\mbox{\fontsize{8}{10}\selectfont $u_3$}$} ;
\draw[xshift=110 cm, >=stealth ] [->] (18.6,-18.3)--(18.6,-27.5); 
\draw [xshift=110 cm](20.5,-26)   node {$\mbox{\fontsize{8}{10}\selectfont $y_3$}$} ;
%%%%%%%%%%%%%%%%%%%%%%%%%%%%%%%%%%%%%%%%%%
\draw[xshift=110 cm, >=stealth ] [->] (4.5,1.5)--(4.5,7)--(30.5,7)--(30.5,1.5); 
\draw[xshift=110 cm, >=stealth ] [->] (32,0)--(35,0)--(35,-29)--(19,-29);
\draw[xshift=110 cm, >=stealth ] [->] (16,-29)--(0,-29)--(0,0)--(3,0);
\end{tikzpicture}
\caption{The Plant $G$ and the Decentralized (Diagonal) Controller (left) versus a ``Ring'' Networked Controller (right)}
\label{Motivation} 
\end{figure}

Note that, in general, the impact of observed variables on each other, represented by the DSF, and the interconnection between subsystems, can be quite different structures.  This is because the states internal to one subsystem are always distinct from another, while the states internal to component systems in the DSF may be shared with other components.  Nevertheless, the point in this example, that the DSF, as a factorization of a system's transfer function, captures an important notion of structure, is always true.  Details about the distinctions between subsystem structure and the signal structure described by the DSF can be found in \cite{Enoch}.

\subsection{Motivation and Scope of Work} In this paper we look at Dynamical Structure Functions from a control systems perspective.  A long standing problem in control of LTI systems was synthesis of decentralized stabilizing controllers (\cite{Wang}) which means imposing on the controller's transfer function matrix $K(s)$ to have a diagonal sparsity pattern.   Quite different to the decentralized paradigm, the ultimate goal of our research would be a systematic method of designing controllers that can be implemented as a LTI network with a pre--specified topology. This is equivalent with computing a stabilizing  controller $K(s)$ whose DSF  $\big( Q(s),P(s) \big)$ satisfies certain  sparsity constraints \cite{Anurag}. So, instead of imposing sparsity constraints on the transfer function of the controller as it is the case in decentralized control, we are interested in imposing the sparsity constraints on the controller's DSF. This would eventually lead to the possibility of designing controllers that can be implemented as a LTI network, see for example Figure~\ref{Motivation}.

\subsection{Contribution}

The contribution of this paper  is the establishment of the intrinsic connections between the DSFs and the left coprime factorizations of a given transfer function and  to give a systems theoretic meaning to sparsity patterns of coprime factors using DSFs.    The importance of this is twofold. First, this is the most common scenario in control engineering practice ({\em e.g.} manufacturing, chemical plants) that the given plant is made out of many interconnected sub--systems. The {\em structure} of this interconnection is captured by a DFS description of the plant which in turn might translate to  left coprime factorization of the plant that features certain sparsity patterns on its factors. This sparisty might be used for the synthesis of a controller to be implemented over a LTI network. Conversely,  in many  applications  it is desired that the stabilizing controller be implemented in a distributed manner, for instance as a LTI network with a pre--specified topology. This is  equivalent to imposing certain sparsity constraints on the left coprime factorization of the controller (via the celebrated Youla parameterization). In order to fully exploit the power of the DSFs approach to tackle these types of problems, we find it useful to underline its links  with the classical notions and results in control theory of LTI systems.  We provide here a comprehensive exposition of the elemental connections between the Dynamical Structure Functions and the Coprime Factorizations of a given Linear Time--Invariant (LTI) system, thus opening the way between exploiting the structure of the plant via the DSF and employing the celebrated Youla parameterization for feedback output stabilization. 

\subsection{Outline of the Paper} In the second Section of the paper we give a brief outline of the theoretical concept of Dynamical Structure Functions as originally introduced in \cite{Sean08}. In the third Section, we show that while  the DSF representation of a given LTI system $L(s)$ is in general never coprime, a closely related representation dubbed a viable $(W,V)$ pair associated with $L(s)$ is always coprime. We also provide the class of all viable $(W,V)$ pairs associated with a given $L(s)$. The fourth Section contains the main results of the paper and it makes a complete explanation of the natural connections between  the DSFs and the viable $(W,V)$ pairs associated with a given $L(s)$ and its left coprime factorizations. The last Section contains the conclusions and future research directions. In the Appendix~A we have provided a short primer on realization theory for improper TFMs which is indispensable for the proofs of the main results.  The proofs of the main results have been placed in Appendix~B.

 \section{Dynamical Structure Functions} \label{adoua}

The main object of study here is a LTI system, which in the continuous--time case are described by the state equations

\begin{subequations}
\begin{align} 
\dot{ x} (t) &=  A x (t) +  Bu (t) ; \; {x}(t_o)={x}_o & \label{ss0ab} \\
y(t) &=  C  x (t) + D u(t) & \label{ss0c}
\end{align}
\end{subequations}

\noindent  where  $A, B, C, D$ are $n\times n$, $n \times m$, $p \times n$, $p \times m$
real matrices, respectively  while  $n$ is also called {\em the order} of the realization. 
Given any $n$--dimensional state--space representation (\ref{ss0ab}), (\ref{ss0c})  of  a LTI system $(A$, $B$, $C$, $D)$, its input--output representation  is given by the {\em Transfer Function Matrix} (TFM) which is the $p\times m$ matrix with real, rational functions entries denoted with

\begin{equation}
\label{state-space}
L(\la) =  \ba{c|c}A & B \\ \hline C & D \ea \overset{def}{=} D + C(\la I_n - A)^{-1}B,
\end{equation}

\begin{rem}\label{lambda} Our results apply on both continuous or discrete time LTI systems, hence we assimilate the undeterminate
  $\lambda$ with  the complex variables $s$ or $z$ appearing in the
Laplace or Z--transform, respectively, depending on the type of the system. 
\end{rem}

For elementary notions in linear systems theory,
such as state equivalence, controlability, observability, detectability, we refer to \cite{Wonham}, or any other standard text book on LTI systems.

By $\mathbb{R}^{p\times m}$ we denote the set of $p \times m$ real matrices and by $\mathbb{R}(\la)^{p\times m}$ we denote  $p \times m$ transfer function matrices (matrices having entries real--rational functions).

This  section contains a discussion  based on reference \cite{Sean08} on the definition of the  Dynamical Structure Functions associated with a LTI system. We start with the given system $L(\la)$ described by the  following state equations, of order $n$:

\begin{subequations}
\begin{align} 
\dot{\tilde x} (t) &= \tilde A\tilde x (t) +\tilde Bu (t) ; \; \tilde{x}(t_o)=\tilde{x}_o & \label{ss1ab} \\
y(t) &= \tilde C\tilde x (t) & \label{ss1c}
\end{align}
\end{subequations}

\begin{assumption} \label{Regularity} {\em (Regularity)} We make the  assumption that the  $\tilde C$ matrix  from (\ref{ss1c}) has full row rank (it is surjective).
\end{assumption}
    
We choose any matrix $\bar C$ such that $T\overset{def}{=} \ba{c} \tilde C \\ \bar C \ea$ is nonsingular (note that such $\bar C$ always exists because $\tilde C$ has full row rank) and  apply  a  state--equivalence transformation

\begin{equation} \label{similea}
 x(t)=T {\tilde x} (t), \quad  A = T\tilde A T^{-1}, \quad  B = T\tilde B, \quad
 C = \tilde C T^{-1}.
\end{equation}

\noindent on  (\ref{ss1ab}),(\ref{ss1c}) in order to get

\begin{subequations}
\begin{align} 
\ba{c} y(t) \\ z (t) \ea & = T \tilde x(t) & \label{ss20} \\
\ba{c} \dot{y}(t)  \\ \dot{z}(t) \ea  &= \ba{cc}  A_{11} & A_{12} \\ A_{21} & A_{22}  \ea  \ba{c} y (t)  \\ z (t) \ea  + \ba{c} B_1 \\ B_2 \ea u(t); \; \; \ba{c} y(t_o) \\ z (t_o) \ea = \ba{c} y_o \\ z_o \ea   &  \label{ss2ab}  \\
y (t) &= \ba{cc} I_p & O \ea \ba{c} y(t)  \\ z(t) \ea   & \label{ss2c}
\end{align}
\end{subequations}

 \begin{assumption} \label{Observability} {\em (Observability)}  We can assume without any loss of generality that the pair $(\tilde C, \tilde A)$ from (\ref{ss1ab}), (\ref{ss1c}) or equivalently the  pair $( A_{12}, A_{22})$ from (\ref{ss2ab}) are observable.
\end{assumption}

\begin{rem} \label{Leuenberger} The argument that the observability assumption does not imply any loss of generality, is connected with the Leuenberger reduced order observer.
\end{rem}

Looking at the Laplace or Z--transform of the equation in (\ref{ss2ab}), we get

\begin{equation} \label{Lapla}
\ba{cc} \la I_p -A_{11} & - A_{12} \\ - A_{21} &\la I_{n-p} - A_{22}  \ea  \ba{c} Y(\la)  \\ Z(\la) \ea =  \ba{c} B_1 \\ B_2 \ea U(\la) 
\end{equation}

By multiplying (\ref{Lapla})  from the left   with the following factor $\Omega(\la)$

\begin{equation} \label{Omega_i}
\Omega (\la) = \ba{cc} I_{p} & \: A_{12} ( \la I_{n-p} - A_{22})^{-1} \\ 
                                      O          & I _{n-p} \\ \ea
\end{equation}

\noindent (note that $\Omega(\la)$ is  always invertible as a TFM) we get 

\begin{equation} \label{WV1}
\ba{cc} \Big((\la I_p - A_{11}) -  A_{12} (\la I_{n-p} - A_{22})^{-1}A_{21} \Big)\;  & \; O \\
* & * \ea \ba{c} Y(\la) \\ Z(\la) \ea = \Omega(\la)  \ba{c} B_1 \\ B_2 \ea U(\la)
\end{equation}

\noindent where the $*$ denote entries  whose exact expression is not needed now. 
Immediate calculations yield that the first block--row in (\ref{WV1}) is equivalent with

\begin{equation} \label{WVin1}
\la\: Y(\la) = \Big( A_{11} +  A_{12} (\la I_{n-p} - A_{22})^{-1}A_{21} \Big) Y(\la)+ \Big(B_1 +  A_{12} (\la I_{n-p} - A_{22})^{-1}B_2 \Big) U(\la)
\end{equation}

\noindent and by making the notation

\begin{subequations}
\begin{align}
 W(\la) & \overset{def}{=}  -A_{11} -  A_{12} (\la I_{n-p} - A_{22})^{-1}A_{21} &  \label{Wprimo} \\
V(\la) & \overset{def}{=}  B_1 +  A_{12} (\la I_{n-p} - A_{22})^{-1}B_2 &  \label{Vprimo}
\end{align}
\end{subequations}

\noindent  we finally get the following equation which describes the relationship between manifest variables

\begin{equation} \label{WVin5} \la Y(\la) = W(\la) Y(\la) + V(\la) U(\la). 
\end{equation}

\begin{rem} \label{Ali} Note that if $V(\la)$ is identically zero, while $W(\la)$ is a constant matrix  having the sparisity of a graph's Laplacian, then (\ref{WVin5}) becomes the free evolution equation $\la Y(\la) = W Y(\la)$. These types of equations  have been extensively studied in cooperative control \cite{Ali} to describe the dynamics of a large group of autonomous agents. Equation  (\ref{WVin5})  can be looked at as a generalization of that model and will be studied here in a different context. 
\end{rem}

Since $L(\la)$ is the input--output operator from $U(\la)$ to $Y(\la)$, we can write  equivalently that  $\displaystyle L(\la) = \big(\la I_p -W(\la)\big)^{-1}V(\la)$,  which is exactly the $(W,V)$ representation from  \cite[(3)/ pp.1671]{Sean08}. (Note that  since $W(\la)$ is always proper it follows that $\big(\la I_p -W(\la)\big)$ is always invertible as a TFM.) Next, let $D(\la)$ denote the TFM obtained by taking the diagonal entries of $W(\la)$, that is $\displaystyle D(\la) \overset{def}{=} \mathrm{diag} \{ W_{11}(\la), W_{22}(\la) \dots W_{pp}(\la) \}$. Then we can write $\displaystyle L(\la) = \Big[ \Big( \la I_p -D(\la) \Big) - - \Big( W(\la) -D(\la) \Big) \Big]^{-1}V(\la)$, or equivalently (note that $(W-D)$ has zeros on the diagonal entries)

\begin{equation} \label{spreQP}
L(\la) = \Big[ I- \Big( \la I_p -D(\la) \Big)^{-1} \Big( W(\la) -D(\la) \Big) \Big]^{-1} \Big( \la I_p -D(\la) \Big)^{-1}  V(\la)
\end{equation}
\noindent and after introducing the notation

\begin{subequations} 
\begin{align}
 Q(\la) & \overset{def}{=}   \Big( \la I_p -D(\la) \Big)^{-1} \Big( W(\la) -D(\la) \Big) & \label{Q} \\
P(\la) & \overset{def}{=}  \Big( \la I_p -D(\la) \Big)^{-1}  V(\la) &  \label{P} 
\end{align}
\end{subequations}

\noindent we get that $\displaystyle L(\la) = \Big( I_p -Q(\la)\Big)^{-1}P(\la)$ or equivalently that

\begin{equation} \label{DSF}
Y(\la)= Q(\la) Y(\la) + P(\la) U(\la)
\end{equation}

\begin{rem} \label{AliObs}
The splitting and the ``extraction'' of the diagonal in (\ref{Q}) are made in order  to make the $Q(\la)$ have the sparsity (and the meaning) of the {\em adjacency matrix} of the graph describing the causal relationships between the manifest variables $Y(\la)$. Consequently, $Q(\la)$ will always have zero entries on its diagonal.
\end{rem}

\begin{defn} \cite[Definition~1]{Sean08}\label{DSFdefin} Given the state--space realization (\ref{ss2ab}),(\ref{ss2c}) of $L(\la)$ the Dynamical Structure Function of the system is defined to be the pair $\big( Q(\la), P(\la) \big)$, where $ Q(\la), P(\la)$ are given by (\ref{Q}) and (\ref{P}) respectively. %Note that since $L(\la)$ is strictly proper, so are $Q(\la)$ and $P(\la)$. 
\end{defn}

\section{Dynamical Structure Functions Revisited}

  One scope of this  paper and also one of its contributions is to emphasize the idea that for a given TFM $L(\la)$ there exist more than one pair $\big( Q(\la), P(\la) \big)$ than the one in (\ref{Q}),(\ref{P}) (originally introduced in \cite{Sean08}) and which  satisfy (\ref{DSF}).  In fact there exists a whole class of pairs $\big( Q(\la), P(\la) \big)$ that do satisfy (\ref{DSF}) and for which  $Q(\la)$ has all its block--diagonal entries equal to zero. In order to illustrate this we need to slightly reformulate the original Definition~\ref{DSFdefin}  of Dynamical Structure Functions associated with a $L(\la)$  as follows:

\begin{defn} \label{DSFdefinw} Given a TFM $L(\la)$, we define a Dynamical Structure Function representation of $L(\la)$ to be any two TFMs $Q(\la)\in\mathbb{R}^{p \times p}(\la)$ and $P(\la)\in\mathbb{R}^{p \times m}(\la)$ with $Q(\la)$ having zero entries on its diagonal,  such that $\displaystyle L(\la) = \big( I_p -Q(\la) \big)^{-1}P(\la)$ or equivalently
\begin{equation} \label{DSFreluat}
Y(\la)= Q(\la) Y(\la) + P(\la) U(\la)
\end{equation}
\end{defn}

The following definition will also be needed in the sequel.

\begin{defn}\label{obiectu} Given the TFM $L(\la)$,  we call a {\em viable} $\big( W(\la),V(\la) \big)$ {\em pair} associated with $L(\la)$, any two TFMs $W(\la)\in\mathbb{R}^{p \times p}(\la)$ and $V(\la)\in\mathbb{R}^{p \times m}(\la)$, with $W(\la)$ having McMillan degree at most $(n-p)$ and such that 

\begin{equation} \label{WV}
 L(\la) = \Big( \la I_p -W(\la) \Big)^{-1}V(\la).
\end{equation}
\end{defn}

\begin{prop} \label{conectia}  Given  a TFM $L(\la)$ then for any given  {\em viable} $\big( W(\la),V(\la) \big)$ {\em pair} associated with $L(\la)$, there exists a {\em unique} DSF representation $\big( Q(\la), P(\la) \big)$ of $L(\la)$ given by (\ref{Q}) and  (\ref{P}), where $\displaystyle D(\la) \overset{def}{=} \mathrm{diag} \{ W_{11}(\la), W_{22}(\la) \dots W_{pp}(\la) \}$ is uniquely determined by $W(\la)$.
\end{prop}
\begin{proof} The proof follows immediately from the very definitions (\ref{Q}), (\ref{P}).
\end{proof}

\begin{rem} \label{sprs} It is important to  remark here that any {\em viable} $\big( W(\la),V(\la) \big)$ pair has the same sparsity pattern with its subsequent DSF representation $\big( Q(\la), P(\la) \big)$. For example $W(\la)$ is lower triangular if and only if $Q(\la)$ is lower triangular. Similarly, for instance $V(\la)$ is tridiagonal if and only if $P(\la)$ is tridiagonal.
\end{rem}

\begin{rem} \label{Main1} Using Proposition~\ref{conectia} we can conclude that in order to find  all DSFs (according to Definition~\ref{DSFdefinw}) associated with a given $L(\la)$, it is sufficient to study the set of all {\em viable} $\Big( W(\la),V(\la) \Big)$ {\em pairs} associated with $L(\la)$. The following theorem gives  closed--formulas for the parameterization of the class of all {\em viable} $\Big( W(\la),V(\la) \Big)$ {\em pairs} associated with a given TFM.
\end{rem}

\begin{theorem} \label{Main} Given a TFM $L(\la)$ having a state--space realization (\ref{ss1ab}),(\ref{ss1c}), we compute any equivalent realization (\ref{ss2ab}),(\ref{ss2c}). The class of all {\em viable} $\Big( W(\la),V(\la) \Big)$ {\em pairs} associated with $L(\la)$ is then given by

\begin{equation} \label{Wdef}
-W(\la)= \ba{c|c}
(A_{22}+ K A_{12})  -\la I_{n-p}  & A_{22}K + KA_{12}K -K A_{11} - A_{21}  \\ 
 \hline  A_{12}   & -A_{11} + A_{12}K  \\ \ea
\end{equation}

\begin{equation} \label{Vdef}
 V(\la)= \ba{c|c} (A_{22}+ K A_{12})  -\la I_{n-p}  & KB_1+B_2 \\ \hline A_{12} & B_1 \\
 \ea 
\end{equation}
\noindent where the $K$ is any matrix  in $\mathbb{R}^{(n-p) \times p}$ and $A_{11},A_{12},,A_{21},A_{22},B_1,B_2$ are as in (\ref{ss2ab}),(\ref{ss2c}).
\end{theorem}
\begin{proof} See Appendix~B.
\end{proof}

\begin{rem} \label{stable}
We remark here the poles of both $W(\la)$ and $V(\la)$ can be allocated at will in the complex plane, by a suitable choice of the matrix $K$ and the assumed observability of the  pair $(A_{12}, A_{22})$ (Assumption~\ref{Observability}).
\end{rem}

\section{Main Results}

The ultimate goal of this line of research would be computing controllers whose DSF has a certain structure. This would allow us for instance to compute controllers that can be implemented as a ``ring'' network (see Figure~\ref{Ring}) or as a ``line'' network  which is important for motion control  of vehicles moving in a platoon formation. However, classical results in LTI systems control theory, such as the celebrated Youla parameterization (or its equivalent formulations) render the expression of the stabilizing controller as a stable  coprime factorization of its transfer function.  As a first step towards employing Youla--like methods for the  synthesis of controllers featuring structured DSF, we need to understand the connections between the stable left coprime factorizations (of a given stabilizing controller) and its DSF representation. We address this problem in this section.
 
 \subsection{A Result on Coprimeness} In this subsection we prove that (by chance rather than by design) for any viable  $\big(W(\la),V(\la) \big)$ pair associated with a given $L(\la)$ (with $W(\la)$ and $V(\la)$ as in Theorem~\ref{Main}) it follows that  $\displaystyle  \big(\la I_p -W(\la)\big), V(\la)\big)$  is a {\em left coprime factorization} of $L(\la)$.  An equivalent condition for  $\big(\la I_p -W(\la),V(\la) \big)$ to be left coprime is for the compound transfer function matrix  
\begin{equation} \label{compusa}
 \ba{cc} \Big( \la I_p -W(\la) \Big) & \; V(\la) \ea
\end{equation}
\noindent to have no (finite or infinite)  Smith zeros (see \cite{Rose70, V, SIMAX} for equivalent characterizations of left coprimeness). Coprimeness  is especially important for output feedback stabilization, since  classical results such as the celebrated Youla parameterization, require a  coprime factorization of the plant while also rendering coprime factors of the stabilizing controllers.   %In this section we will also elaborate on the connections between {\em stable} left coprime factorizations and the viable $\Big(\la I_p-W(\la),V(\la) \Big)$ pairs associated with a given $T(\la)$.

\begin{assumption} \label{Controlability}
{\em (Controllability)} From this point onward we assume that the realization (\ref{ss1ab}),(\ref{ss1c}) of $L(\la)$ is controllable.
\end{assumption} 
 
 \begin{theorem} \label{CupruMin} Given a TFM $L(\la)$, then for any viable  $\big(W(\la),V(\la) \big)$ pair associated with a given $L(\la)$ (with $W(\la)$ and $V(\la)$ as in Theorem~\ref{Main}) it follows that  $\displaystyle \big(\la I_p -W(\la),V(\la)\big)$  is a {\em left coprime factorization} of $L(\la)$.
 \end{theorem}
 \begin{proof} See Appendix~B.
 \end{proof}

  \begin{rem} \label{nevercop} We remark here that while any viable  $\big(W(\la),V(\la) \big)$ pair associated with a given $T(\la)$ makes out for a left coprime factorization $L(\la)=\big( \la I_p- W(\la)\big)^{-1}V(\la)$, the DSF $\displaystyle L(\la) = \big( I_p -Q(\la)\big)^{-1}P(\la)$ are in general never coprime (unless the plant is stable or diagonal). That is due to the fact that in general not all the unstable zeros of $\big(\la I_p - D(\la)\big)$ cancel out when forming the products in (\ref{Q}), (\ref{P}) and the same unstable zeros will result in poles/zeros cancelations when forming the product $\displaystyle L(\la) = \big( I_p -Q(\la)\big)^{-1}P(\la)$.
\end{rem}
  
\subsection{Getting from DSFs to Stable Left Coprime Factorizations}
 In this subsection we show that for any viable pair $\big( W(\la), V(\la) \big)$ with both $W(\la)$ and $V(\la)$, respectively being stable, there exists a class of stable left coprime factorizations. Furthermore, there exists a class of stable left coprime factorizations that preserve the sparsity pattern of the original viable pair $\big( W(\la), V(\la) \big)$. 
 
 Note that for any viable pair $\big(W(\la), V(\la) \big)$ is an {\em improper} rational function and it has exactly $p$ poles at infinity of multiplicity one, hence the $\big( \la I_p - W(\la) \big)$ factor (the denominator of the factorization) is inherently {\em unstable} (in either continuous or discrete--time domains).  We remind the reader  that  any the poles of both $W(\la)$ and $V(\la)$ can be allocated at will in the stability domain (Remark~\ref{stable}). In this subsection, we show how to get  from viable pair  $\big(W(\la),V(\la) \big)$ of $L(\la)$  in which both factors $W(\la)$ and $V(\la)$ are stable, to a stable left coprime factorization  $L(\la) = M^{-1}(\la)N(\la)$. We achieve this  without altering any of the stable poles of $W(\la)$ and $V(\la)$ (which are the modes of $(A_{22}+K A_{12})$ in (\ref{Wdef}), (\ref{Vdef}))  and while at the same time keeping the McMillan degree to the minimum. The problem is  to displace the $p$ poles at infinity (of multiplicity one) from the $\big (\la I_p-W(\la)\big)$ factor. To this end we will use the {\em Basic Pole Displacement Result} from \cite[Theorem~3.1]{SIMAX} that shows that this can be achieved by premultiplication with an adequately chosen invertible factor $\Theta(\la)$ such that when forming the product $\displaystyle \Theta(\la)\big(\la I_p-W(\la) \big)$ all the $p$ poles at infinity of the factor $\big(\la I_p-W(\la)\big)$ cancel out. Here follows the precise statement:

   \begin{lem}\label{step2} Given a viable pair  $\big(\la I_p-W(\la),V(\la) \big)$ of  $L(\la)$ then for any  
 
 \begin{equation} \label{Theta}
\Theta(\la)\overset{def}{=} \ba{c|c} A_x  -\la I_p & T_4 \\
   \hline T_5   & O \ea
 \end{equation}
 \noindent with $A_x,T_4,T_5$ arbitrarily chosen such that $A_x$ has only stable eigenvalues and  both $T_4,T_5$ are invertible, it follows that
 
 \begin{equation}
  \ba{cc} M(\la) & \; N(\la) \ea\overset{def}{=} \Theta(\la) \ba{cc} \Big( \la I_p -W(\la) \Big) & \; V(\la) \ea
 \end{equation}
 
 \noindent is  a {\em stable} left coprime factorization $L(\la)=  M^{-1}(\la) N(\la) $. Furthermore, 
  \begin{small}
 \begin{equation} \label{MN}
  \ba{cc} M(\la) & \; N(\la) \ea =  \ba{cc|cc} A_x  -\la I_{n-p}\; \; & T_4 A_{12}  & (A_xT_4 -T_4 A_{11} +T_4A_{12}K) \; \; & T_4B_1 \\
                           O & A_{22}+KA_{12} -\la I_p & (A_{22}K+KA_{12}K-KA_{11}-A_{21})  & KB_1+B_2 \\ 
                               \hline T_4^{-1} & O & I  & O \\
 \ea 
 \end{equation}
 \end{small}
 \noindent hence all the modes  in $(A_{22}+K A_{12})$ (which are the original  stable poles of  $W(\la)$ and $V(\la)$) are preserved in the $M(\la)$ and $N(\la)$ factors.
 \end{lem}

 \begin{proof} See Appendix~B.
 \end{proof}

 \begin{rem} \label{buna} We remark that  for {\em any} diagonal $A_x$ having only stable eigenvalues $\Theta(\la)=(\la I_p - A_x)^{-1}$ yields
% \[
%(\la I_p - A_x)^{-1} \Big(\big(( \la I_p-A_x) -(W(\la)-A_x) \big),V(\la) \Big) =
 %\]
 % \[
%\Big(\big(I_p-(\la I_p - A_x)^{-1} (W(\la)-A_x) \big),(\la I_p - A_x)^{-1} V(\la) \Big) 
 %\]
 %\noindent which is 
 a stable left coprime factorization of $L(\la)$ that preserves the {\em  sparsity structure} of the initial viable $\Big(\la I_p-W(\la),V(\la) \Big)$ pair.
 \end{rem}

 \subsection{Connections with the Nett \& Jacobson Formulas \cite{Nett}}

 In this subsection, we are interested in connecting the expression from (\ref{MN}) for the pair  $(M(\la), N(\la))$  to the classical result of state--space derivation of  left coprime factorizations of a given plant originally presented in \cite{Nett} (and generalized in \cite{Lucic}).

 \begin{prop} \cite{Nett, Lucic}
\label{SIMAX99}
Let $L(\la)$ be an arbitrary $m\times p$ TFM and $\Omega$ a domain in $\boC$. The class of {\em all} left coprime factorizations of $L(\la)$ over $\Omega$, $T(\la) = {M}^{-1}(\la ){N}(\la)$, is given by
\bb{dc2}
\ba{cc} M(\la) &  N(\la) \ea
 = U^{-1}\ba{c|cc} (A - FC) - \la I & -F & B \\ \hline
C & I & O \ea,
\ee

\noindent where $A, B, C, F$ and $U$ are real matrices accordingly dimensioned such that

\noindent {\bf \em i)}  $U$ is any $p \times p$ invertible matrix,

\noindent {\bf \em ii)}  F is any feedback matrix that  allocates the observable modes of the $(C,A)$ pair to $\Omega$,

\noindent {\bf \em iii)} $L(\la) =  \ba{c|c}A - \la I & B \\ \hline C & O \ea$ is a stabilizable realization.

 \end{prop}
 
Due to Assumption~\ref{Controlability}, we have to replace  the stabilizability from point {\bf \em iii)} with a controlability assumption. We start off with $L(\la)$ given by the equations (\ref{ss1ab}),(\ref{ss1c})
 
 \begin{equation} \label{oooo1}
 L(\la)=\ba{cc|c}  A_{11} - \la I_p&  A_{12}  &B_1 \\ 
                             A_{21}  & A_{22} - \la I_{n-p}   & B_2 \\ \hline
                             I & O & O \\ \ea  
 \end{equation}

 \noindent  and we want to retrieve  (\ref{MN}) by using the  parameterization in Proposition~\ref{SIMAX99}. First apply a state-equivalence  $\displaystyle T=\ba{cc} T_4 & O \\ K & I \ea$ in order to get 
 
 \begin{equation} \label{oooo1}
L(\la) = \ba{cc|c}T_4 (A_{11}- A_{12} K)T_4^{-1}-\la I_p  & T_4 A_{12}  & T_4 B_1 \\ 
                           \big( KA_{11} + A_{21} -KA_{12}K -A_{22}K \big)T_4^{-1}  &KA_{12}+ A_{22} - \la I_{n-p}  & K B_1+B_2 \\  \hline
                             T_4^{-1} & O & O \\ \ea  
 \end{equation}
 
 Next, we only need to identify the $F$ feedback matrix from point  {\bf \em ii)} of  Proposition~\ref{SIMAX99}, which in this case is proven to be  given by
 \begin{equation} \label{F}
 F=\ba{cc} (T_4^{-1}A_x T_4) - A_{11} + A_{12}K
 \\ -K(T_4^{-1}A_x T_4) + A_{22}K - A_{21} \ea
 \end{equation}

 To check,  simply plug  (\ref{F}) in (\ref{dc2}) for the realization  (\ref{oooo2})  of $L(\la)$.

  \subsection{Getting from the Stable Left Coprime Factorization to the DSFs}
 
 In this subsection we  show that for {\em almost} every stable left coprime factorization of a given LTI system, there is an associated a {\em unique} viable  $\big(W(\la),V(\la) \big)$ pair and consequently (via Remark~\ref{conectia}) a unique DSF representation $\big(Q(\la),P(\la) \big)$. The key role in establishing this one to one correspondence is played  by a non--symmetric Riccati equation, whose solution existence is a generic property.  This result is meaningful, since for controller synthesis while we are interested in the DSF of the controller, in general we only have access to a stable left coprime of the controller.

 We start with a  given  stable left coprime factorization (\ref{dc2}) for  $L(\la)$ having an order $n$ realization
 
 \begin{equation} \label{uuu}
\ba{cc} M(\la) &  N(\la) \ea
 = U^{-1}\ba{c|cc} (A - FC) - \la I & -F & B \\ \hline
C & I & O \ea
 \end{equation}

\noindent  to which we  apply a type (\ref{similea}) state--equivalence transformation with $T\in \mathbb{R}^{n \times n}$   such that $\displaystyle CT^{-1} = \ba{cc} I_p & O \ea$. Note that such a $T$ always exists because of Assumption~\ref{Regularity}. It follows that (\ref{uuu}) takes the form

 \begin{equation} \label{Kontroller}
\ba{cc} M(\la) & N(\la) \ea =  \ba{cc|cc} A_{11} +F_1  -\la I_p \; \; & A_{12}  \; \; & F_1& B_1 \\
                           A_{21} +F_2 & A_{22} -\la I_{n-p} & F_2 & B_2 \\ 
                               \hline I_p & O  & I_p & O \\  \ea
 \end{equation}

\noindent and denote

\begin{equation} \label{aplus}
A^+ \overset{def}{=}\ba{cc} A_{11} +F_1  & - A_{12} \\  -(A_{21} +F_2) &  A_{22}  \ea
\end{equation}

The solution to the following nonsymmetric algebraic Riccati matrix equation is paramount to the main result of this subsection, since it underlines the one to one correspondence between (\ref{Kontroller}) and its  {\em unique} associated  viable  $\big(W(\la),V(\la) \big)$ pair.
\begin{prop} \label{Ric}
The nonsymmetric algebraic Riccati matrix equation
  
  \begin{equation} \label{Riccati}
  K(A_{11}+F_1) -A_{22}K    -KA_{12}K  +  (A_{21} + F_2) = O
  \end{equation}

\noindent has a stabilizing solution $K$ ({\em i.e.} $(A_{11}+F_1-A_{12}K)$ is stable) if and only if the $A^+$ matrix from (\ref{aplus}) has a stable invariant subspace of dimension $p$   with basis matrix

\begin{equation} \label{V}
\ba{c}V_1 \\  V_2 \ea
\end{equation}

\noindent having $V_1$ invertible ({\em i.e.} disconjugate). In this case $K=V_1^{-1}V_2$ and it is the {\em unique} solution of (\ref{Riccati}). 
\end{prop}

\begin{proof} It follows from \cite{EJC}.
\end{proof}

\begin{rem} \label{Cristian} Since in our case $A^+$ is stable, all its invariant subspaces are actually stable
(including the whole space). Therefore, the Riccati equation has a stabilizing solution if and
only if the matrix $A^+$ has an invariant subspace of dimension $p$ which is disconjugate. Hence,
if for example $A^+$ has only simple eigenvalues, the Riccati equation always has a
solution (we can always select $p$ eigenvectors (from the n eigenvectors) to form a disconjugate
invariant subspace). In this case, all we have to do is to order the eigenvalues in a Schur form
such that the corresponding invariant subspace has $V_1$ invertible. Although this is a {\em generic
property}, when having Jordan blocks of dimension greater than one it might happen that the
matrix $A^+$ has no disconjugate invariant subspace of appropriate dimension $p$, and therefore
the Riccati equation has no solution (stable or otherwise).
\end{rem}

\begin{theorem} \label{inversa}
 Given any stable left coprime factorization $L(\la)=M^{-1}(\la)N(\la)$ and its state--space realization (\ref{Kontroller}),  let $K$ be the solution  of the nonsymmetric algebraic Riccati equation (\ref{Riccati}) and denote $A_x\overset{def}{=}(F_1 + A_{11} - A_{12}K)$. Then, a state--space realization for $\displaystyle \ba{cc} M(\la) & N(\la) \\ \ea$ is given by

 \begin{small}
 \begin{equation} \label{MN1}
  \ba{cc} M(\la) & \; N(\la) \ea =  \ba{cc|cc} A_x  -\la I_{n-p}\; \; &  A_{12}  & (A_x - A_{11} +A_{12}K) \; \; & B_1 \\
                           O & A_{22}+KA_{12} -\la I_p & (A_{22}K+KA_{12}K-KA_{11}-A_{21})  & KB_1+B_2 \\ 
                               \hline I & O & I  & O \\
 \ea. 
 \end{equation}
 \end{small}

\noindent Furthermore, from (\ref{MN1})  we can recover the exact expression of the subsequent viable $\Big( W(\la),V(\la) \Big)$ {\em pair} associated with $L(\la)$, where $W(\la)$ and $V(\la)$ are given by (\ref{Wdef}) and (\ref{Vdef}), respectively.
 \end{theorem}
 \begin{proof} For the proof, simply plug
     \begin{equation}  \label{FRicc}
 \ba{c}  F_1 \\ F_2\ea\overset{def}{=} \ba{cc} A_x  - A_{11} + A_{12}K
 \\ -KA_x + A_{22}K - A_{21} \ea
 \end{equation}
\noindent into the expression of (\ref{Kontroller}) in order to obtain (\ref{MN1}). The rest of the proof follows from Lemma~\ref{step2}, by taking $T_4$ to be equal with the identity matrix $I_p$.
 \end{proof}
 
 \begin{rem} We remark here that in general there is no correlation between the sparsity pattern of the stable left coprime (\ref{Kontroller}) we start with and its associated viable $\big( W(\la),V(\la) \big)$ {\em pair} produced in Theorem~\ref{inversa}. That is to say that the converse of the observation made in Remark~\ref{buna} is not valid. This poses additional problems for controller synthesis, since it might happen to encounter stable left coprime factorizations that have no particular sparsity pattern (are dense TFMs) while their associated viable $\big( W(\la),V(\la) \big)$ {\em pair} are sparse. This is due to the fact that in general, the $A_x$ matrix in Theorem~\ref{inversa}  can be a dense matrix. One  way to circumvent this problem would be to use a carefully adapted version of Youla's parameterization in which the stable left coprime factorization to be replaced with a DSF description where both  with $\big( W(\la),V(\la) \big)$ factors are stable. This is the topic of our future investigation.
 \end{rem}
 
\section{Conclusions} 
 In this paper we have presented an exhaustive discussion on the intrinsic  connections between the DSFs associated with a given transfer function and its left coprime factorizations. We have showed that rather than dealing directly with the DSF representation it is more beneficial to work on the so--called viable $\big(W(\la),V(\la) \big)$ pairs associated with a given system. This theoretical results ultimately aim at a method of designing LTI controllers that can be implemented over a network with a pre--specified topology. We currently have sufficient conditions for the existence of such controllers but we miss the necessary conditions. While in general these conditions might be very hard to find, we expect  to find such conditions for plants featuring special DSF structures.

\section*{Appendix A}

\begin{defn} \label{improper} A TFM $L(\la)$ is called {\em improper} if for at least one  of its entries (which are real--rational functions), it holds that the degree of the numerator is strictly larger than the degree of the denominator. 
\end{defn}

\begin{prop} (\cite{Verg79, Verg81}) \label{realization}
Any improper (even polynomial) $p \times m$ rational matrix $L(\la)$ with coefficients in $\FF$  
 has a descriptor
realization of the form 
\bb{a1}
L(\la) = D +  C(\la E - A)^{-1}B =: \ba{c|c} A - \la E & B \\ \hline C & D \ea, 
\ee
where $A, E \in \FF^{n \times n}$, $B \in \FF^{n \times m}$, $C \in \FF^{p \times n}$, $D \in \FF^{p \times m}$, and the so called {\em  pole pencil} $A - \la E$ is {\em regular}, 
i.e., it is square  and det$(A - \la E )  \not\equiv 0$. 
The dimension $n$ of the square matrices $A$ and $E$ is called the {\em order of 
the realization} \rf{a1}.
\end{prop}

%\begin{defn} We use $\La(A- \la  E)$ to denote the union of generalized eigenvalues of the  regular pencil $A - \la E$ (finite and infinite, multiplicities counting).
%Occasionally, we shall also use the more compact notation 
%$L(\la) = (A-\la E,B,C,D)$ to denote  \rf{a1}.   
%\end{defn}

\begin{defn} The descriptor realization \rf{a1} of $L(\la)$ is called {\em minimal} if 
its order is as small as possible among all realizations of this kind. 
\end{defn}

%\begin{prop}(\cite{Verg81})  \label{m1} Given a TFM $L(\la)$, a descriptor realization \rf{a1} of $L(\la)$  is minimal if and only if all the following conditions hold true
%\begin{subequations}
%\label{minim}
%\begin{eqnarray} 
 %\rank \ba{cc} A - \la E & B\ea  &=& n, \quad
%\forall \la \in \CC, \\
%\rank \ba{cc}  E & B \ea  &=& n, \\ 
%\rank \ba{c} A - \la E \\ C \ea &=& n, \quad
%\forall \la \in \CC, \\
%\rank \ba{c} E \\ C \ea &=& n, \\ 
%A \ker(E)  &\subseteq &  \Im(E).  
%\end{eqnarray}
%\end{subequations}
%The conditions of minimality \rf{minim} are usually known as finite and
%infinite controllability,  finite and infinite observability, and absence of nondynamic modes, respectively.  
%\end{prop}

\begin{defn} \label{McMillanDesc} The {\em McMillan degree}  of $L(\la)$ -- denoted 
$\delta(L)$ -- is the sum  of the orders of all the poles of $L(\la)$ (finite and infinite). 
\end{defn}

%\begin{prop}
%For a minimal descriptor realization \rf{a1} of order $n$ we have 
%$\delta(G) = \rank E \leq n $.  
%\end{prop}

\begin{rem} \label{impo} The principal inconvenience of realizations of the form \rf{a1} is 
that their minimal possible order is greater than the McMillan 
degree of $L(\la)$, unless $L(\la)$ is proper, and this brings important 
technical difficulties in factorization problems in which the McMillan degree 
plays a paramount role.  A remedy to this 
is to use a generalization of \rf{a1} in which either the ``$B$" or the ``$C$" matrix is
replaced by a matrix pencil, as stated in the next Proposition. 
\end{rem}

\begin{prop} (\cite{SIMAX}) Any  improper $p \times m$ TFM  $L(\la)$ has a realization 
\bb{a2}
L(\la) = \ba{c|c} A - \la E & B - \la F \\ \hline C & D \ea \overset{def}{=}  D + C(\la E - A)^{-1}(B - \la F), 
\ee
and for any fixed $\alpha, \beta \in \FF$, not both zero, there exists a realization 
\bb{a3}
L(\la) =  \ba{c|c} A - \la E & B(\alpha - \la \beta)\\ \hline C & D \ea\overset{def}{=} D + C(\la E - A)^{-1}B(\alpha - \la \beta), 
\ee 
where $A, E \in \FF^{n \times n}$, $B, F \in \FF^{n \times m}$, $C \in \FF^{p \times n}$, 
$D \in \FF^{p \times m}$, and the  pole pencil $A - \la E$ is regular.  
A realization  \rf{a3} will be called centered at $\frac{\alpha}{\beta}$ 
(if $\beta = 0$ we interpret $\frac{\alpha}{\beta}$ as $\infty$).      
Occasionally, we shall use also the more compact notation 
$L(\la) = (A-\la E,B - \la F ,C,D)$ and $L(\la)= (A-\la E,B(\alpha - \la \beta) ,C,D)$ to denote \rf{a2} and \rf{a3}, respectively. Realizations of type \rf{a3}
 have been dubbed  {\em pencil realizations}.
\end{prop}

\begin{defn}(\cite{SIMAX}) \label{m3} We call  realizations of the type \rf{a2} or \rf{a3} {\em minimal} if the
dimension of the square matrices $A$ and $E$ (also called the order of the realization) 
is as small as possible  among all realizations of the respective kind. 
\end{defn}

\begin{prop} (\cite{SIMAX})  \label{m4} Any TFM $L(\la)$ has a minimal realization of type
\rf{a2} of order equal to $\delta(L)$. For any fixed $\alpha$ and $\beta$, not both zero,  
and such that  $\frac{\alpha}{\beta}$ is not a pole of $L(\la)$ there also exists a minimal 
realization of type \rf{a3} of order equal to $\delta(L)$. 
The condition imposed on
$\frac{\alpha}{\beta}$ is needed only for writing down minimal realizations \rf{a3} which have 
order equal to  $\delta(L)$. More precisely, 
even if $\frac{\alpha}{\beta}$ is a pole of $L(\la)$ we can still write a realization 
\rf{a3} but the minimal order will with necessity be greater than $\delta(L)$. 
This is exactly what is happening for realizations \rf{a1} which are obtained 
from \rf{a3} for $\alpha = 1$ and $\beta = 0$, and for which the minimal order 
is necessary greater than $\delta(L)$, provided 
$\frac{\alpha}{\beta} = \infty $ is a pole of $L(\la)$. Notice that for 
\rf{a3} we can always choose freely $\alpha$ and $\beta$ such as 
to ensure $\frac{\alpha}{\beta}$ is not a pole of $L(\la)$. For the rest of the paper, 
if not otherwise stated, we assume this choice implicitly. 
The nice feature of \rf{a2} and \rf{a3} that their minimal order  
equals the McMillan degree of $L(\la)$ recommends them for the kind of problems 
treated in this paper.     
\end{prop}

\begin{prop} (\cite{SIMAX}) \label{m5} A given
realization of type \rf{a2} of a TFM $L(\la)$ is minimal if and only if all of the following conditions hold true  
\begin{subequations}
\label{minim2}
\begin{eqnarray}
\rank \ba{cc} A - \la E & B - \la F \ea  &=& n, \quad
\forall \la \in \CC,  \\ 
\rank \ba{cc}  E & F \ea  &=& n, \\ 
\rank \ba{c} A - \la E \\ C \ea &=& n, \quad
\forall \la \in \CC,  \label{w1} \\ 
\rank \ba{c} E \\ C \ea &=& n, \label{w2}
\end{eqnarray}
\end{subequations}
while for realizations of type \rf{a3} similar conditions  result by 
simply replacing (a) and (b) in \rf{minim2} with 
\begin{subequations}
\label{minim3} 
\begin{eqnarray}
\rank \ba{cc} A - \la E & B (\alpha - \la \beta) \ea  &=& n, \quad
\forall \la \in \CC, \label{w3} \\ 
\rank \ba{cc}  E & B \ea  &=& n. \label{w4} 
\end{eqnarray}
\end{subequations}
\end{prop}

\begin{prop} Any  two minimal realizations 
$L(\la) = (A-\la E,B(\alpha - \la \beta),C,D)$ and
$L(\la) = (\wti A-\la \wti E,\wti B(\alpha - \la \beta),\wti C,\wti D)$ are
always related by an equivalence transformation as  
\begin{equation} \label{DescEquiv} \wti E = QEZ, \quad \wti A = QAZ, \quad \wti B = QB, \quad
\wti C = CZ, \quad \wti D = D,
\end{equation}
where $Q$ and $Z$ are unique invertible matrices.  
\end{prop}

\section*{Appendix B}

{\bf Proof of Theorem~\ref{Main}}  We prove that any pair $\displaystyle \Big( W(\la), V(\la) \Big)$ given by (\ref{Wdef}),(\ref{Vdef}) satisfies (\ref{WV}). We start with the equations (\ref{Lapla})

\begin{subequations} 
\begin{align}
\ba{cc} \la I_p -A_{11} & - A_{12}  \\ 
                             -A_{21}  &\la I_{n-p} - A_{22} \\  \ea \ba{c} Y(\la)  \\ Z(\la)  \ea &=  \ba{c} B_1 \\ B_2 \ea U(\la) & \label{oo1} \\
 Y(\la) = \ba{cc} I_p & O\ea \ba{c} Y(\la)  \\Z(\la)   \ea & & \label{oo2}
\end{align}
\end{subequations}

\noindent and apply a type (\ref{similea}) state equivalence transformation  with

\begin{equation}  \label{T}
T=\ba{cc}  I_p & O \\
                   K & I_{n-p}  \ea 
\end{equation}

\noindent where  $K$ can be  any matrix  in $\mathbb{R}^{(n-p) \times p}$, in order to get 

\begin{small}
\begin{subequations} 
\begin{align}
 \ba{cc} \la I - (A_{11} -A_{12}K)  &  -A_{12} \\  (- KA_{11} - A_{21} + A_{22}K +KA_{12}K)  \;\; & \la I -(A_{22}+KA_{12})  \ea \ba{c} Y(\la)  \\ K Y(\la) +Z(\la)   \ea   & = \ba{c} B_1 \\ KB_1+B_2  \ea U(\la) &  \label{intermedia} \\
  Y(\la)= \ba{cc} I_p  & O \ea \ba{c} Y(\la) \\   K Y(\la) +Z(\la)    \ea & & \label{intermedia2}         
\end{align}  
\end{subequations}
 \end{small}

 \noindent respectively.  In a similar manner with getting from (\ref{Lapla})  to (\ref{WV1}) via (\ref{Omega_i}),  we multiply (\ref{intermedia}) to the left with the following invertible factor
 
 \begin{equation} \label{OmegaInterim}
 \Omega_K (\la) = \ba{ccc} I_{p} & \; \;  A_{12} \Big(\la I_{n-p} - (A_{22}+KA_{12})\Big)^{-1}  \\ 
                                           O   &     I _{n-p}    \ea
 \end{equation}
 
 \noindent After the multiplication is performed, the first block row of the resulting equation yields $\displaystyle \big(\la I_p - -W(\la)\big) Y(\la)=V(\la)U(\la)$ which is exactly (\ref{WV}) with $W(\la)$ and $V(\la)$ having the expressions in (\ref{Wdef}) and (\ref{Vdef}), respectively. Finally, from the expression of $W(\la)$ in (\ref{Wdef}), clearly the McMillan degree of $W(\la)$ cannot exceed $(n-p)$.

%``$\supset$'' We prove that any viable pair $\displaystyle \big( W(\la), V(\la) \big)$  satisfying (\ref{WV}) must be of the form (\ref{Wdef}),(\ref{Vdef}).  {\em Abridged version of proof to be added by Serban.}

%%%%%%%%%%%%%%%%%%%%%%%%%%%%%%%%%%%%%%%%%%%%%%%%%%%%%%%%%%%%%%%%%%%%%%%%%%%%%%%%%%%
%%%%%%%%%%%%%%%%%%%%%%%%%%%%%%%%%%%%%%%%%%%%%%%%%%%%%%%%%%%%%%%%%%%%%%%%%%%%%%%%%%%
%%%%%%%%%%%%%%%%%%%%%%%%%%%%%%%%%%%%%%%%%%%%%%%%%%%%%%%%%%%%%%%%%%%%%%%%%%%%%%%%%%%
%%%%%%%%%%%%%%%%%%%%%%%%%%%%%%%%%%%%%%%%%%%%%%%%%%%%%%%%%%%%%%%%%%%%%%%%%%%%%%%%%%%
%%%%%%%%%%%%%%%%%%%%%%%%%%%%%%%%%%%%%%%%%%%%%%%%%%%%%%%%%%%%%%%%%%%%%%%%%%%%%%%%%%%
%%%%%%%%%%%%%%%%%%%%%%%%%%%%%%%%%%%%%%%%%%%%%%%%%%%%%%%%%%%%%%%%%%%%%%%%%%%%%%%%%%%
%%%%%%%%%%%%%%%%%%%%%%%%%%%%%%%%%%%%%%%%%%%%%%%%%%%%%%%%%%%%%%%%%%%%%%%%%%%%%%%%%%%
%%%%%%%%%%%%%%%%%%%%%%%%%%%%%%%%%%%%%%%%%%%%%%%%%%%%%%%%%%%%%%%%%%%%%%%%%%%%%%%%%%%
%%%%%%%%%%%%%%%%%%%%%%%%%%%%%%%%%%%%%%%%%%%%%%%%%%%%%%%%%%%%%%%%%%%%%%%%%%%%%%%%%%%
%%%%%%%%%%%%%%%%%%%%%%%%%%%%%%%%%%%%%%%%%%%%%%%%%%%%%%%%%%%%%%%%%%%%%%%%%%%%%%%%%%%
%%%%%%%%%%%%%%%%%%%%%%%%%%%%%%%%%%%%%%%%%%%%%%%%%%%%%%%%%%%%%%%%%%%%%%%%%%%%%%%%%%%

{\bf Proof of Theorem~\ref{CupruMin}}  An equivalent condition for the pair $\big(\la I_p -W(\la),V(\la) \big)$ to be coprime (over the compactification of $\mathbb{C}$) is for the compound transfer function matrix  
\begin{equation} \label{locala1}
 \ba{cc} \Big( \la I_p -W(\la) \Big) & \; V(\la) \ea
\end{equation}
\noindent to have no (finite or infinite)  Smith zeros (see \cite{Rose70, V, SIMAX} for equivalent characterizations of left coprimeness). According to \cite[Theorem~2.1]{SIMAX} (see also \cite{Rose70, Verg79}) the Smith zeros of (\ref{locala1}) are among the Smith zeros  (generalized eigenvalues) of the system--pencil of any minimal realization of (\ref{locala1}). Hence we break this proof in two distinct parts: in part {\bf I)} we compute a type (\ref{a3}) {\em pencil realization} for (\ref{locala1}) and prove that is indeed minimal, in the sense of Definition~\ref{m5}. In part {\bf II)} of the proof we show that the system-pencil of the minimal realization from part {\bf I)} has no finite of infinite Smith zeros (generalized eigenvalues).

 {\bf I)} We will show that the following type (\ref{a3}) {\em pencil realization} for $\displaystyle \ba{cc} \big( \la I-W(\la) \big) & \; V(\la) \ea$ is a {\em minimal} realization in the sense of Definition~\ref{m5}:
 
 \begin{small}
 \begin{equation} \label{coprime}
\ba{cc} \big( \la I-W(\la) \big) & \; V(\la) \ea =   \ba{cc|cc} (A_{22}+ K A_{12})  -\la I_{n-p}\; \; & O  & (A_{22}K + KA_{12}K -K A_{11} - A_{21}) \; \; & KB_1+B_2 \\
                           O & I_p & I_p (\la_o - \la)  & O \\ 
                               \hline A_{12} & I_p & \la_o I_p -A_{11} + A_{12}K  & B_1 \\
 \ea 
 \end{equation}
\end{small}

\noindent  {\bf I a)} {\em Observability  for any finite $\la \in \mathbb{C}$}  We note that
 \[
 \ba{cc} (A_{22}+ K A_{12})  -\la I_{n-p}\; \; & O \\
 O  & I_p \\
  A_{12} & I_p \ea =   \ba{ccc} I & O & K \\ O & I & -K \\ O & O & I  \ea \ba{cc} A_{22}  -\la I_{n-p}\; \; & O \\
 O  & I_p \\
  A_{12} & I_p \ea
 \]

 \noindent where the right hand side has full column rank for any $\la \in \mathbb{C}$, due to the observability of the pair $(A_{12}, A_{22})$ (from Assumption~\ref{Observability}).  Hence point (\ref{w1}) of Definition~\ref{m5} holds via the Popov--Belevitch--Hautus (PBH) criterion.
 
 \noindent   {\bf I b)} {\em Observability  at $\la = \infty$} is equivalent via  point (\ref{w2}) of Definition~\ref{m5} with the following matrix having full column rank  
 \[
 \ba{cc} I & O \\ O & O \\ A_{12} & I \ea.
 \]
 
 \noindent {\bf I c)} {\em Controllability  for any finite $\la \in \mathbb{C}$} We look at the following succession of equivalent singular matrix pencils 
 
 \[
 \ba{cccc} (A_{22}+ K A_{12})  -\la I_{n-p}\; \; & O  & (A_{22}K + KA_{12}K -K A_{11} - A_{21}) \; \; & KB_1+B_2 \\
                           O & I_p & I_p (\la_o - \la)  & O \ea \sim
 \]
 
  \[
 \ba{cccc} (A_{22}+ K A_{12})  -\la I_{n-p}\; \; & O  & (\la K  -K A_{11} - A_{21}) \; \; & KB_1+B_2 \\
                           O & I_p & I_p (\la_o - \la)  & O \ea \sim
 \]
 
   \[
 \ba{cccc} (A_{22}+ K A_{12})  -\la I_{n-p}\; \; & K  & (\la_o K  -K A_{11} - A_{21}) \; \; & KB_1+B_2 \\
                           O & I_p & I_p (\la_o - \la)  & O \ea \sim
 \]
 
    \[
 \ba{cccc} (A_{22}+ K A_{12})  -\la I_{n-p}\; \; & K  & (-K A_{11} - A_{21}) \; \; & KB_1+B_2 \\
                           O & I_p &-\la  I_p   & O \ea \sim
 \]
 
     \[
 \ba{cccc} (A_{22}+ K A_{12})  -\la I_{n-p}\; \; & K  &  - A_{21} \; \; & KB_1+B_2 \\
                           O & I_p &-\la  I_p + A_{11}  & O \ea \sim
 \]
 
      \[
 \ba{cccc} (A_{22}+ K A_{12})  -\la I_{n-p}\; \; & K  &  - A_{21} \; \; & \; \; B_2 \\
                           O & I_p &-\la  I_p + A_{11}  &\; \; -B_1 \ea \sim
 \]

\begin{equation} \label{locala4}
 \ba{cccc} A_{22} -\la I_{n-p}\; \; & K  &  - A_{21} \; \; & \; \; B_2 \\
                           -A_{12} & I_p &-\la  I_p + A_{11}  &\; \; -B_1 \ea \sim \ba{cccc} \la I_p -A_{11} & - A_{12} & B_1 & -I_p\\
                            -A_{21}  &\la I_{n-p} - A_{22}  & B_2  & K\\  \ea 
\end{equation}

 \noindent The full row rank of the last pencil above for any $\la \in \mathbb{C}$, follows from the controlability Assumption~\ref{Controlability} and the PBH criterion and it fulfills point (\ref{w3}) of Definition~\ref{m5}.

 \noindent {\bf I d)} {\em Controlability  at $\la = \infty$:}  is equivalent via  point (\ref{w4}) of Definition~\ref{m5} with the following matrix having full row rank
 \[
 \ba{cccc} I_p & O & O & O \\ O & O & I_{n-p} & O \ea.
 \]

 {\bf II)} We look at the system--pencil of the realization (\ref{coprime}), namely
 
 \begin{small}
 \begin{equation} \label{mare}
 \mathcal{S}(\la) \overset{def}{=} \ba{cccc} (A_{22}+ K A_{12})  -\la I_{n-p}\; \; & O  & (A_{22}K + KA_{12}K -K A_{11} - A_{21}) \; \; & KB_1+B_2 \\
                           O & I_p & I_p (\la_o - \la)  & O \\ 
                              A_{12} & I_p & \la_o I_p -A_{11} + A_{12}K  & B_1 \\
 \ea 
 \end{equation}
 \end{small}
 
\noindent  We will show next that the singular pencil in (\ref{mare}) has no finite or infinite Smith zeros (generalized eigenvalues), which will conclude that the pair $\Big(\la I-W(\la),V(\la) \Big)$ is left coprime. We will show this, by proving that $\mathcal{S}(\la)$ keeps full row rank for any $\la \in \mathbb{C}$ and also for $\la=\infty$.
 
\noindent {\bf II a)} {\em No Finite Smith Zeros} We look at the following succession of equivalent matrix pencils
 \[
   \ba{cccc} (A_{22}+ K A_{12})  -\la I_{n-p}\; \; & O  & (A_{22}K + KA_{12}K -K A_{11} - A_{21}) \; \; & KB_1+B_2 \\
                           O & I_p & I_p (\la_o - \la)  & O \\ 
                              A_{12} & I_p & \la_o I_p -A_{11} + A_{12}K  & B_1 \\
 \ea  \sim
 \]
 
  \[
   \ba{cccc} (A_{22}+ K A_{12})  -\la I_{n-p}\; \; & O  & (\la K -K A_{11} - A_{21}) \; \; & KB_1+B_2 \\
                           O & I_p & I_p (\la_o - \la)  & O \\ 
                              A_{12} & I_p & \la_o I_p -A_{11}   & B_1 \\
 \ea  \sim
 \]
  
    \[
   \ba{cccc} A_{22}  -\la I_{n-p}\; \; & O  &- A_{21} \; \; & B_2 \\
                           O & I_p & I_p (\la_o - \la)  & O \\ 
                              A_{12} & I_p & \la_o I_p -A_{11}   & B_1 \\
 \ea  \sim    \ba{cccc} A_{22}  -\la I_{n-p}\; \; & O  &- A_{21} \; \; & B_2 \\
                              A_{12} & I_p & \la_o I_p -A_{11}   & B_1 \\
                                                         O & I_p & I_p (\la_o - \la)  & O \\ 
 \ea  \sim
 \]

  \[
     \ba{cccc} A_{22}  -\la I_{n-p}\; \; & O  &- A_{21} \; \; & B_2 \\
                              A_{12} & O & \la I_p -A_{11}   & B_1 \\
                                                         O & I_p & I_p (\la_o - \la)  & O \\ 
 \ea  \sim \ba{cccc}  A_{11} - \la I_p &  A_{12} & B_1 &O \\
                             A_{21}  &  A_{22} - \la I_{n-p}   & B_2 &O \\  
                             \la I_p & O & O & I_p \ea
 \]

\noindent The last pencil above clearly holds  full row rank for any $\la \in \mathbb{C}$ due to Assumption~\ref{Controlability} and the PBH criterion.
 
 \noindent {\bf II b)} {\em No Smith Zeros at Infinity:} Follows by the adaptation of  \cite[Lemma~1]{Verg79}. %{\em To be added by Serban.} The proof ends.

  {\bf Proof of Lemma~\ref{step2}} This proof is based entirely on \cite[Theorem~3.1]{SIMAX} (Basic Pole Displacement Result).  We start with the following type (\ref{a2}) {\em minimal} realization of 
 \begin{small}
  \begin{equation} \label{IWV}
  \ba{cc} \big(\la I_p - W(\la)\big) & V(\la)\ea = \ba{cc|cc}  I_p & O & I_p (\la_o - \la)  & O \\ 
O &  (A_{22}+ K A_{12})  -\la I_{n-p}\; \;  & (A_{22}K + KA_{12}K -K A_{11} - A_{21}) \; \; & KB_1+B_2 \\
                               \hline  I_p &  A_{12} & \la_o I_p -A_{11} + A_{12}K  & B_1 \\
 \ea 
 \end{equation}
 \end{small}
  
  It can be observed that (\ref{IWV}) is already in the ordered block-Schur form \cite[(2.14)/pp. 252]{SIMAX}. We want to employ \cite[Theorem~3.1]{SIMAX} in order to compute the invertible TFM from \cite[(3.1)/pp. 252]{SIMAX} which we denote with $\Theta(\la)$ that by premultiplying (\ref{IWV}) will cancel out the $p$ poles at infinity of (\ref{IWV}). Any type (\ref{a3}) realization of a valid $\Theta(\la)$  satisfies   \cite[(3.2)/pp. 252]{SIMAX} for certain {\em invertible} $X$ and $Y$ matrices. Hence for any 
  \[
  \Theta(\la)= \ba{c|c} A_x-\la I_p & B_x(\la -\la_o) \\ \hline
  C_x & D_x \\ \ea
  \]  

 \noindent (with $D_x$ must be invertible because $\Theta(\la)$ is invertible ) we write the conditions from  \cite[(3.2)/pp. 252]{SIMAX}  which are equivalent with 

\begin{equation} \label{ert}
\ba{cc} A_x-\la I_p & B_x(\la -\la_o) \\  C_x & D_x \ea \ba{c} X \\ I \ea = \ba{c} Y \\ O \ea (I_p -\la O)
\end{equation}

\noindent From the first block row of (\ref{ert}) we get that $C_x X=-D_x$  and from the second block-row of (\ref{ert}) we get $B_x(\la-\la_o)=Y-(A_x-\la I_p)X$. Consequently 

  \[
  \Theta(\la)= \ba{c|c} A_x-\la I_p & Y-(A_x-\la I_p)X \\ \hline
  C_x & -C_xX \\ \ea
  \]  

\noindent which is equivalent with 

  \[
  \Theta(\la)= \ba{c|c} A_x-\la I_p & Y \\ \hline
  C_x & O \\ \ea
  \]  

\noindent where $C_x$ and $Y$ are arbitrary invertible matrices. We have denoted $C_x$ with $T_4$ and we have denote $Y$ with $T_5$ to avoid notational confusion. The proof ends.
\end{document}